\documentclass{article}

\usepackage{hyperref}
\usepackage{amsthm,amssymb,amsmath}

\author{Ryan Schwartz \and J\'ozsef Solymosi}

\newtheorem{theorem}{Theorem} % 1st argument is your name for it
\newtheorem{lemma}[theorem]{Lemma}     % 2nd argument is what is printed

\newtheorem{conjecture}[theorem]{Conjecture}

\title{Combinatorial Applications of the Subspace Theorem\thanks{%
Work by J\'ozsef Solymosi was supported by NSERC, ERC-AdG. 321104, and OTKA NK 104183 grants. }}

\begin{document}

\maketitle

\section{Introduction}
\label{sec:intro}

The Subspace Theorem is a powerful tool in number theory.  It has appeared in 
various forms and been adapted and improved over time.  It's applications
include diophantine approximation, results about integral points on algebraic curves and
the construction of transcendental numbers.  But its usefulness extends beyond
the realms of number theory.  Other applications of the Subspace Theorem include
linear recurrence sequences and finite automata.  In fact, these structures are
closely related to each other and the construction of transcendental numbers.

The Subspace Theorem also has a number of remarkable combinatorial applications.
The purpose of this paper is to give a survey of some of these applications 
including sum-product estimates and bounds on unit distances.  The presentation 
will be from the point of view of a discrete mathematician.  We
will state a number of variants of the Subspace Theorem below but we will
not prove any of them as the proofs are beyond the scope of this work.  
However we will give a proof of a simplified special case of the
Subspace Theorem which is still very useful for many problems in discrete
mathematics.

A number of surveys have been given of the Subspace Theorem highlighting its
multitude of applications.  Notable surveys include those of Bilu \cite{Bilu07},
Evertse and Schlickewei \cite{Ever99} and Corvaja and Zannier \cite{Corv08}.
These give many proofs of results from number theory and algebraic geometry
using the Subspace Theorem including those mentioned above.

Wolfgang M.~Schmidt was the first to state and prove a variant of the Subspace Theorem
in 1972 \cite{Schm72}. His theorem was extended and played a very important role
in modern number theory. Before we state the Subspace Theorem we need some definitions.  A
\emph{linear form} is an expression of the form
$L(x)=a_1x_1+a_2x_2+\dots+a_nx_n$ where $a_1, \dots, a_n$ are constants and
$x=(x_1, \dots, x_n)$.  A collection of linear forms is \emph{linearly
   independent} if none of them can be expressed as a linear combination of the
   others.  Given $x=(x_1, \dots, x_n)$ we define the maximum norm \[\|x\| =
      \max(|x_1|, \dots, |x_n|).\]
\begin{theorem}[Subspace Theorem I]
   \label{thm:schmidt}
   Suppose we have $n$ linearly independent linear forms $L_1, L_2, \dots, L_n$
   in $n$ variables with algebraic coefficients.  Given $\varepsilon>0$, the
   non-zero integer points $x=(x_1, x_2, \dots, x_n)$ satisfying \[|L_1(x)L_2(x)\dots
      L_n(x)| < \|x\|^{-\varepsilon}\] lie in finitely many proper linear
      subspaces of $\mathbb{Q}^n$.

\end{theorem}
It generalised the Thue-Siegel-Roth Theorem on the approximation of algebraic
numbers \cite{Roth55} to higher dimensions.

Theorem~\ref{thm:schmidt} has been extended in various directions by many
authors including Schmidt himself, Schlickewei, Evertse, Amoroso and Viada.
Analogues have been proved using $p$-adic norms and over arbitrary number fields
and bounds on the number of subspaces required have been found.  These bounds
depend on the degree of the number field and the dimension.  For some of these
results and more information see \cite{Ever99}, \cite{Ever02} and \cite{Amor09}.

Now we give a $p$-adic version of the Subspace Theorem that we will use in the next
section.  Given a prime $p$, the $p$-adic absolute value is denoted $|x|_p$ and 
satisfies
$|p|_p=1/p$.  $|x|_{\infty}$ denotes the usual absolute
value so $|x|_{\infty}=|x|$.  We may refer to $\infty$ as the \emph{infinite 
prime}.  We define the height of a rational vector $x=(x_1, \dots, x_n)$ by 
\[H(x) =
\prod_{p} \|x\|_p = \prod_{p}\max\{1, |x_1|_p, \dots, |x_n|_p\}.\]  Here the 
product extends over all primes including the infinite prime.  Note that for any 
$x$ only finitely many terms in the product are not $1$.
\begin{theorem}[Subspace Theorem II] \label{thm:ss2}
   Suppose $S = \{\infty, p_1, \dots, p_t\}$ is a finite set of primes, 
   including the
   infinite prime.  For every $p\in S$ let $L_{1,p},
   \dots, L_{n,p}$ be linearly independent linear forms in $n$ variables with
   algebraic coefficients.  Then for any $\varepsilon > 0$ the solutions $x\in
   \mathbb{Z}^n$ of \[\prod_{p\in S}\prod_{i=1}^n|L_{i,p}(x)|_v \le
   H(x)^{-\varepsilon}\] are contained in finitely many proper linear subspaces
   of $\mathbb{Q}^n$.
\end{theorem}

The power and utility of the Subspace Theorem is already evident in the above
forms but
there is a corollary, often itself called the Subspace Theorem, which makes even more
applications possible.  This corollary was originally given by Evertse,
Schlickewei and Schmidt \cite{Ever02}.  We present the version with the best known
bound due to Amoroso and Viada \cite{Amor09}.
\begin{theorem}[Subspace Theorem III] \label{thm:ss3} Given an algebraically
   closed field $K$ and a subgroup $\Gamma$ of $K$ of finite rank $r$, suppose 
   $a_1, a_2, \dots, a_n\in K^*$.  Then the number of solutions of the equation 
   \begin{equation}
      a_1z_1+a_2z_2+\dots+a_nz_n=1 \label{eq:sum}
   \end{equation}
   with $z_i\in\Gamma$ and no subsum on the left hand side vanishing is at most 
   \[A(n,r)\le (8n)^{4n^4(n+nr+1)}.\]
\end{theorem}

The Erd\H{o}s unit distance problem is an important problem in combinatorial
geometry.  It asks for the maximum possible number of unit distances between $n$
points in the plane.  This problem is still open but recently Frank de Zeeuw and 
the authors have made progress towards this problem when the distances
considered come from certain groups.

The structure of this paper will be as follows.  In the next section we give a
number of well-known applications of the Subspace Theorem.  In
Section~\ref{sec:comb} we give combinatorial applications.  In particular,
Section~\ref{subsec:mann} contains the special case of the Subspace Theorem via
Mann's Theorem, Section~\ref{subsec:unit} gives unit distance bounds and
Section~\ref{subsec:sumProduct} gives sum-product estimates.

\subsection*{Acknowledgements}

The work in Section~\ref{subsec:mann} is joint work with Frank de Zeeuw. The
authors are thankful to Jarik Ne\v set\v ril for the encouragement to write this
survey. We are also thankful to the organizers of the workshop in Pisa,
"Geometry, Structure and Randomness in Combinatorics", where the parts of this
paper were presented.

\section{Number theoretic applications}
\label{sec:nt}

\subsection{Transcendental numbers}

Adamczewski and Bugeaud showed that all irrational automatic numbers are
transcendental using the Subspace Theorem.  An \emph{automatic number} is a
number for which there exists an integer $b>0$ such that when the number is
written in $b$-ary form it is the output of a finite automaton with input the
natural numbers written from right to left.  For more details see \cite{Adam07}
or \cite{Bilu07}.

Here we will use a method similar to the proof of Theorem~3.3 in \cite{Bilu07}
to show:
\begin{theorem}\label{thm:trans}
   The number $\alpha$ given by the infinite sum \[\alpha = \sum_{n\ge 1} \frac{1}{2^{2^n}}\] is transcendental.
\end{theorem}
Kempner showed in the early twentieth century that a large
class of numbers defined similarly to $\alpha$ are transcendental \cite{Kemp16}.
The Subspace Theorem provides a tidy proof of this fact.

\begin{proof}[Proof of Theorem~\ref{thm:trans}]
   Consider the binary expansion: \[\alpha =
   \frac{1}{4}+\frac{1}{16}+\frac{1}{256}+\frac{1}{65536}+\dots =
0.0101000100000001\hdots_2.\]  So the
binary expansion of $\alpha$ consists of sections of zeros of increasing length
separating solitary ones.  Thus the expansion is not periodic and hence $\alpha$
is not rational.  We let $b_n$ be the string given by the first $n$ digits of
this expansion.  One can check that each $b_n$ has two disjoint substrings of
zeros of length $n/8$.

Assume $\alpha$ is not transcendental.  Then it is algebraic.  Now each $b_n$
starts with a string $AOBO$, where $O$ is a string of zeroes, the length of $O$ 
is at least $n/8$ and $A$
and $B$ might have length zero.  We will use the rational number represented in
base $2$ by
$AOBOBO\hdots$ to approximate $\alpha$.  Call this number $\pi$.  Then \[\pi =
\frac{M}{2^a(2^b-1)}\] where $M\in\mathbb{Z}$ and $a$ and $b$ are the lengths of
the strings $A$ and $OB$ respectively.  Clearly $b\ge n/8$ and $a+b\le n$ since
$AOB$ is a substring of $b_n$.  Since $\alpha$ starts with $b_n$ we have
\[|\alpha-\pi| \le \frac{1}{2^{a+b+n/8}} \implies
|2^{a+b}\alpha-2^a\alpha-M| \le \frac{1}{2^{n/8}}.\]

Now we apply the Subspace Theorem.  We let $S=\{2, \infty\}$ and 
\begin{align*}
   L_{1,\infty}(x) &= x_1, &L_{2,\infty}(x) &= x_2, &L_{3,\infty}(x) &= \alpha
   x_1-\alpha x_2 - x_3, \\
   L_{1,2}(x) &= x_1, &L_{2,2}(x) &= x_2, &L_{3,2}(x) &= x_3.
\end{align*}
Note that by our assumption that $\alpha$ is not transcendental the linear form 
$L_{3,\infty}$ has algebraic coefficients.  Let $x=(2^{a+b}, 2^{b}, M)$.  Now 
$|M|\le 2^{a+b}$ since $0<\pi<1$.  So $\|x\|
\le 2^{a+b}\le 2^n$.  Multiplying the absolute values of the linear forms
together we get
\begin{align*}\prod_{p\in S}\prod_{i=1}^3|L_{i,p}(x)|\quad &=&
   &|2^a|_2|2^a|_{\infty}|2^{a+b}|_2|2^{a+b}|_{\infty}|M|_2|2^{a+b}\alpha-2^a\alpha-M|_{\infty}
   \\
&\le& &\frac{1}{2^{n/8}} \\ &\le& &\frac{1}{\|x\|^{1/8}} \le 
   H(x)^{-1/8}.\end{align*}  The first two inequalities hold
because $|\alpha-\pi|\le 2^{-a-b-n/8}$ and $|2|_2|2|_{\infty}=1$.

We can do this for each $n$ and $b=b(n)$ increases as $n$ increases since $b\ge
n/8$.  Thus infinitely many of the vectors $x=x(n)$ are distinct.  By
Theorem~\ref{thm:ss2} these vectors are contained in finitely many subspaces of
$\mathbb{Q}^3.$
Thus one of these subspaces contains infinitely many of them.  That is, there exist
$c,d,e\in\mathbb{Q}$ such that \[c2^{a(n)}+d2^{a(n)+b(n)}+eM(n) = 0\] for
infinitely many $n$.  $e$ cannot be zero since $b(n)\to \infty$ as $n\to
\infty$.  Dividing by $2^{a(n)}(2^{b(n)}-1)$ and taking limits we get
$\alpha=-d/e$ so $\alpha$ is rational.  This is a contradiction.  Thus
$\alpha$ must be transcendental.
\end{proof}

\subsection{Linear recurrence sequences}

A linear recurrence sequence is a sequence of numbers where the first few terms
are given and the higher order terms are given by a recurrence relation.  A
famous example is the Fibonacci sequence $\{F_n\}$ where $F_1 = F_2 = 1$ and 
$F_n=F_{n-1}+F_{n-2}$ for $n>2$.  More formally, a \emph{linear
recurrence sequence} consists of constants $a_1, \dots, a_k$ in a field $K$ for 
some $k>0$ along with a sequence $\{R_n\}_{n=1}^{\infty}$ with $R_i\in K$ for $1 
\le i \le k$
 and \[R_n = a_1R_{n-1} + a_2R_{n-2} + \dots + a_kR_{n-k}, \quad \mathrm{for}\ 
 n>k.\]  If
 $\{R_n\}$ is not expressible by any shorter
recurrence relation then it is said to have order $k$.  In this case each 
$a_i\ne 0$.

We are interested in the structure of the zero set of a linear recurrence
sequence.  This is the set \[S(\{R_m\})=\{i\in\mathbb{N} : R_i = 0\}.\]  The
Skolem-Mahler-Lech Theorem 
states that this set consists of the union of finitely many points and arithmetic
progressions.  Schmidt has given a quantitative bound for this theorem using
various tools including the Subspace Theorem.

We will show a special case of this theorem using Theorem~\ref{thm:ss3}.  We
will restrict our attention to simple nondegenerate linear recurrence sequences.
To define such sequences we need to define the companion polynomial
of the recurrence sequence.  If $\{R_n\}$ is given as above then the
\emph{companion polynomial} of $\{R_n\}$ is
$C(x)=x^k-a_1x^{k-1}-\dots-a_{k-1}x-a_k$.  Suppose the roots of this polynomial
are $\alpha_1, \dots, \alpha_{\ell}$ with multiplicity $b_1, \dots, b_{\ell}$
respectively.  Clearly, each $\alpha_i$ is nonzero.  If the companion polynomial 
has $k$ distinct roots it is called
\emph{simple}.  If $\alpha_i/\alpha_j$ is not a root of unity for any $i\ne j$
then the sequence is called \emph{nondegenerate}.
\begin{theorem}
   Suppose $\{R_m\}$ is a simple nondegenerate linear recurrence sequence of 
   order $k$.  Then
   \[|S(\{R_m\})| \le (8k)^{8k^6}.\]
\end{theorem}
\begin{proof}
   We can express the recurrence relation using a matrix equation:
   \[\begin{bmatrix} a_1 & a_2 & \dots & a_{k-1} & a_k \\ 1 & 0 &
   \dots & 0 & 0 \\ 0 & 1 & \dots & 0 & 0 \\ \vdots & \vdots & \ddots & \vdots &
   \vdots
   \\ 0 & 0 & \dots & 1 & 0 \end{bmatrix}^n\begin{bmatrix} R_k \\
   R_{k-1} \\ \vdots \\ R_1 \end{bmatrix} = \begin{bmatrix} R_{k+n}
   \\ R_{k-1+n} \\ \vdots \\ R_{1+n}\end{bmatrix}.\]  We call the matrix above
   $A$.  The characteristic polynomial of $A$ is given by \[\chi(\lambda) =
   \lambda^k-a_1\lambda^{k-1}-\dots-a_k.\]  This is the same as the companion
   polynomial of $\{R_n\}$.  Thus $A$ has distinct nonzero eigenvalues and so
   can be diagonalized.  Thus multiplying out the left hand side we can solve
   for $R_n$ to get 
   \[R_n=c_1\alpha_1^n+c_2\alpha_2^n+\dots+c_k\alpha_k^n\quad\mathrm{for\ every\
   }n>k.\]  Then applying
   Theorem~\ref{thm:ss3} to the equation $c_1x_1+c_2x_2+\dots+c_kx_k=0$ with 
   solutions from the group of rank at most $k$ generated by
   $\{\alpha_1,
   \dots, \alpha_k\}$ we get that the number of solutions is at most \[A(k, k) = 
   (8k)^{4k^4(k+k^2+1)} \le (8k)^{8k^6}.\]
   Since the sequence is nondegenerate we cannot have two values $n, n'$ giving 
   the same value for $\alpha_i^n$ and $\alpha_i^{n'}$ for each $i$, hence each 
   solution corresponds
   to a unique value from $S(\{R_m\})$.
\end{proof}

\section{Combinatorial applications}
\label{sec:comb}

\subsection{A proof of a very special case of the Subspace Theorem}
\label{subsec:mann}

Theorem~\ref{thm:ss3} gives a bound on the number of nondegenerate solutions of
a linear equation from a multiplicative group with rank not too large.  What
happens if the group in question has rank zero.  This corresponds to solutions
that are roots of unity.  The Subspace Theorem can then be seen as a
generalisation of the following results which follows from a theorem of 
H.B.~Mann from 1965.
\begin{theorem} \label{thm:mann}
   Given $(a_1, \dots, a_k)\in\mathbb{Q}^k$, consider the equation
   \[a_1x_1+a_2x_2+\dots+a_kx_k=0.\] The number of solutions $(\omega_1,
   \dots, \omega_k)$ of this equation with the $\omega_i$'s roots of unity and no 
   vanishing subsum is at
   most $(k\cdot \Theta(2k))^k$ where \[\Theta(k)= \mathop{\prod_{p\leq k}}_{p\
   \mathrm{prime}}p.\]
\end{theorem}
Note that the logarithm of the function $\Theta$ above is an important function 
in number theory called the first Chebyshev function.

Theorem~\ref{thm:mann} along with Lemma~\ref{lem:mann} below were proved by 
Frank de Zeeuw and the authors in \cite{Schw12a}.  This theorem provides another 
starting point for the development of the Subspace Theorem.
The roots of unity give a relatively simple example of an infinite 
multiplicative
group.  We will give the proof of Theorem~\ref{thm:mann} below.  First we prove 
Lemma~\ref{lem:mann} which was Mann's result
mentioned above \cite{Mann65}.

\begin{lemma}[Mann] \label{lem:mann}
Suppose we have \[a_1\omega_1+a_2\omega_2+\dots+a_k\omega_k =0,\] with $a_i\in 
\mathbb{Q}$, the
$\omega_i$ roots of unity, and no proper nontrivial subsum vanishing.  Then for 
every $i,j$, $\left(\omega_i/\omega_j\right)^{\Theta(k)} = 1$.
\end{lemma}
\begin{proof}
Dividing by an appropriate factor we may assume that our equation is of the form 
$1 +
a_2\omega_2+\dots+a_k\omega_k=0$.  Then we just need to show that each
$\omega_i^{\Theta(k)}=1$.  We let $m$ be the smallest value such that 
$\omega_i^m=1$ for
$1\le i\le k$.  The proof proceeds by showing that $m$ is squarefree and any 
prime that divides $m$ cannot be larger than $k$.  This means $m\le \Theta(k)$.

Suppose $m = p^j m'$ with $p$ prime and $(p,m') = 1$.
Now we have
$\omega_i = \rho^{\sigma_i}\cdot \omega_i^*,$
with $\rho$ a primitive $p^j$-th root of unity.  So rewriting the sum grouping 
powers of $\rho$ we get
\[0 = 1+ (a_2\omega_2+\dots+a_k\omega_k) = 1+ (
\alpha_0+\alpha_1\rho+\dots+\alpha_{p-1}\rho^{p-1}),\] where, for each $i$, 
$\alpha_i\in K:=\mathbb{Q}(\omega_2^*, \dots, \omega_k^*)$ satisfies
\[\alpha_\ell = \sum_{i\in I_\ell} a_i\omega_i^*,\ \textrm{with}\ I_\ell=\{i : 
\sigma_i=\ell\}.\]  Let $f(x)=\alpha_{p-1}x^{p-1}+\dots+\alpha_1x+(1+\alpha_0)$.  
Then $f$ is a polynomial of degree at most $p-1$
over the field $K$ and $f(\rho)=0$.  If $f$ were identically zero then, by the 
minimality of $m$, we would have a vanishing subsum.

The degree of $\rho$ over $K$ gives us that $p$ divides $m$ only once.
Specifically since $[K(\rho):\mathbb{Q}] = [K(\rho):K][K:\mathbb{Q}]$ we have
\[\deg_K(\rho)=[K(\rho):K] = \frac{[K(\rho):\mathbb{Q}]}{[K:\mathbb{Q}]} = 
\frac{\phi(m)}{\phi(m/p)}.\]  This is $p$ if $j>1$ and $p-1$ if $j=1$.  But the 
degree of $f$ is at most $p-1$ so we must have $j=1$ since $\rho$ is a root of 
$f$.

Now $f$ must be a multiple of the irreducible polynomial $m$ of $\rho$ over $K$.  
But $m(x)=x^{p-1}+x^{p-2}+\dots+1$ so $f(x)=cm(x)$ where $c$ is a nonzero 
constant.  Thus $f$ has $p$ nonzero coefficients and thus so does the original 
sum giving $p\le k$.
\end{proof}

\begin{proof}[Proof of Theorem~\ref{thm:mann}]
   We first show that if we are given $a\in\mathbb{C}^*$ and two sums 
   $a_1\omega_1+\dots+a_k\omega_k = a$ and $a_1'\omega_1'+\dots+a_k'\omega_k'=a$ 
   with rational coefficients and no vanishing subsums then for any $\omega_j'$, 
   there is an $\omega_i$ such that $(\omega_j'/\omega_i)^{\Theta(2k)}=1$.

Since $a_1\omega_1+\dots+a_k\omega_k = a = a_1'\omega_1'+\dots+a_k'\omega_k'$, 
we get $a_1\omega_1+\dots+a_k\omega_k - a_1'\omega_1'-\dots-a_k'\omega_k'=0$.

This sum may have vanishing subsums so we consider minimal vanishing subsums of 
the form
   \[\sum_{i\in I_\ell}a_i\omega_i-\sum_{j\in I_\ell'}a_j'\omega_j' =0.\]
Each $\omega_j'$ is contained in such a minimal
subsum of length at most $2k$.  This subsum also contains some $\omega_i$ 
otherwise the original sum would have a vanishing subsum.  Now the previous 
lemma gives that $(\omega_j'/\omega_i)^{\Theta(2k)}=1$.

Note that above we require $a\in\mathbb{C}^*$.  If $a=0$ then the original sums 
will count as vanishing subsums when we consider the combined equation so 
Lemma~\ref{lem:mann} does not apply.

Now we can prove the theorem.  For $a\in\mathbb{C}^*$ and $k$ a positive 
integer define $S(a,k)$ 
as the set of $k$-tuples $(\omega_1,\dots,\omega_k)$, where each $\omega_i$ is a
root of unity, such that
there are $a_i\in\mathbb{Q}$ satisfying $a_1\omega_1+\dots+a_k\omega_k=a$ with no
vanishing subsums.

We fix a $k$-tuple $(\omega_1,\dots,\omega_k)\in S(a,k)$.  Given an element of
$S(a,k)$, for each 
$\omega_j'$ (the $j$-th coordinate of that element) there is an $i$
such that $\omega_i^{-\Theta(2k)}(\omega_j')^{\Theta(2k)}=1$.  So $\omega_j'$ is a root of the
polynomial $\omega_i^{-\Theta(2k)}x^{\Theta(2k)}=1$.  This polynomial has
$\Theta(2k)$ roots.  We have $k$ choices for $j$ so at most
$k\Theta(2k)$ choices for each $\omega_j'$.  This gives the required bound.
\end{proof}

This theorem can be used to prove Theorem~\ref{thm:unitMann} from the next
section.  We will show how using the Subspace Theorem instead allows the proof
of the stronger Theorem~\ref{thm:unit}.

\subsection{Unit distances}
\label{subsec:unit}

The unit distance problem was first posed by Erd\H{o}s in 1946 \cite{Erdo46}.
It asks for the maximum number, $u(n)$, of pairs of points with the same
distance in a collection of $n$ points in the plane.  By scaling the point set
one may assume that the most popular distance is one, hence the name of the
problem.  The problem seeks asymptotic bounds.  Erd\H{o}s gave a construction
using a $\sqrt{n}\times\sqrt{n}$ portion of a square lattice giving \[u(n) \ge
n^{1+c/\log\log n}.\]  Number theoretic bounds for the number of integer
solutions of the equation $x^2+y^2=a$ give the above inequality.  Erd\H{o}s
conjectured that the magnitude of $u(n)$ is close to this lower bound.  The
best known upper bound is $u(n)\le cn^{4/3}$.  A number of proofs have
been given showing $u(n)\le cn^{4/3}$ using tools such as cuttings, edge
crossings in graphs and the Szemer\'edi-Trotter Theorem.  The first proof was
due to Spencer, Szemer\'edi and Trotter \cite{Spen84}.  For more details of the
problem see \cite{Bras06}.  We will look at a special case of this problem when
the distances considered come from a multiplicative group with rank not too
large.  This does not seem to be a huge limitation as the unit distances from
the lower bound construction above come from such a group as will be explained
below.

Using Theorem~\ref{thm:mann} Frank de Zeeuw and the authors were able to show the
following theorem \cite{Schw12a}.  Two points in the plane are said to have
\emph{rational angle} if the angle that the line between these two points makes
with the $x$-axis is a rational multiple of $\pi$.
\begin{theorem}\label{thm:unitMann}
   Let $\varepsilon>0$.  Given $n$ points in the plane, the number of unit 
   distances with rational
   angle between pairs of
   points is less than $n^{1+\varepsilon}$.
\end{theorem}
These unit distances correspond to roots of unity.  The proof proceeds by
counting certain paths in
the unit distance graph and using Mann's Theorem to bound the number of edges.

Using the Subspace Theorem in place of Mann's Theorem one can instead consider
unit distances from a multiplicative group with rank not too large with respect
to the number of points \cite{Schw13b}.  Note that a unit distance in the plane 
can (and will) be considered as a complex number of unit length.  So all unit 
distances can be considered as coming from a subgroup of $\mathbb{C}^*$.
\begin{theorem}
   \label{thm:unit}
   Let $\varepsilon>0$.  There exist a positive integer $n_0$ and a constant 
   $c>0$ such that given $n>n_0$ points in the plane, the number of unit 
   distances coming from a subgroup $\Gamma \subset \mathbb{C}^*$ with rank 
   $r<c\log n$ is less than $ n^{1+\varepsilon}$.
\end{theorem}
This is our first combinatorial application of the Subspace Theorem.  The proof
is given below.

Suppose $G=G(V,E)$ is a graph on $v(G)=n$ vertices and
$e(G)=cn^{1+\alpha}$ edges.  We denote the minimum degree in $G$
by $\delta(G)$.

Note that by removing vertices with degree less than $(c/2)n^{\alpha}$ we have a
subgraph $H$ with at least $e(H)\ge(c/2)n^{1+\alpha}$ edges and $\delta(H)\ge
(c/2)n^{\alpha}$.  The number of vertices in $H$ is at least
$v(H)=\sqrt{c}n^{1/2+\alpha/2}$.  We will consider such a well behaved subgraph
instead of the original graph.

\begin{proof}[Proof of Theorem~\ref{thm:unit}]
Let $G$ be the unit distance graph on $n$ points with unit distances coming from
$\Gamma$ as edges.  We show that there are less than $n^{1+\varepsilon}$ such
distances, i.e. edges, for any $\varepsilon > 0$.  We can assume that $e(G)\ge
(1/2)n^{1+\varepsilon}, v(G)\ge n^{1/2+\varepsilon/2}$ and $\delta(G)\ge
(1/2)n^{\varepsilon}$.

Consider a path in $G$ on $k$ edges $P_k=p_0p_1\dots p_k$.  We denote by
$u_i(P_k)$ the unit vector between $p_i$ and $p_{i+1}$.  The path is
\emph{nondegenerate} if $\sum_{i\in I}u_i(P_k) = 0$ has no solutions where $I$
is a nonempty subset of $\{0,1,\dots,k-1\}$.  Note that such a sum is a sum of
roots of unity with no vanishing subsums.  We will denote by $\mathcal{P}_k(v,w)$ the set of nondegenerate
paths of length $k$ between vertices $v$ and $w$.

The number of nondegenerate paths of length $k$ from any vertex is at least
\[\prod_{\ell=0}^{k-1}(\delta(G)-2^{\ell}+1) \ge 
\frac{n^{k\varepsilon}}{2^{2k}}.\] The first expression is true since if we
consider a path $P_\ell$ on $\ell<k$ edges then all
but $2^{\ell}-1$ possible continuations give a path $P_{\ell+1}$ with no
vanishing subsums.  The inequality 
is true if we have assume $2^k \le (1/2)n^{\varepsilon}$, which is true 
if $k < \varepsilon\log n/\log 2 - 1$, a fact we will confirm at the end of the
proof.  From this we get that the number of nondegenerate paths $P_k$ in the graph is at least 
$n^{1/2+(k+1/2)\varepsilon}/2^{2k+1}$.
So there exist vertices $v,w$ in $G$ with \[|\mathcal{P}_k(v,w)|\ge 
\frac{n^{(k+1/2)\varepsilon-3/2}}{4^{k}}.\]

Consider a path $P_k\in\mathcal{P}_k(v,w)$, $P_k=p_0p_1\dots p_k$.  Let $a$ be the complex number
giving the vector between $p_0$ and $p_k$.  Since $P_k$ is nondegenerate we get
a solution of $(1/a)x_1+(1/a)x_2+\dots+(1/a)x_k=1$
with no vanishing
subsums.  Thus 
Theorem~\ref{thm:ss3} gives \[|\mathcal{P}_k(v,w)|\le(8k)^{4k^4(k+kr+1)}.\]
This with the lower bound give 
\begin{eqnarray*}
   ((k+1/2)\varepsilon-3/2)\log n &\le& k\log 4 + 4k^4(k+kr+1)\log(8k)\\
                                  &\le& 5rk^5\log k,
\end{eqnarray*}
where the last inequality holds for $k$ large enough.  So
\begin{equation}
   \label{eq:ineq}
   \varepsilon\le 
   \frac{5rk^4\log k}{\log n} + \frac{3}{2k}.
\end{equation}

This inequality holds for $k \ge \exp\left((1/5)W(5c_2\log n/r)\right)$
where $W$ is the positive real-valued function
that solves $x=W(x)e^{W(x)}$.
Note that the function $W$ satisfies $(1/2)\log x\le W(x)\le \log x$ for $x\ge e$.

Since $r+1\le c\log n$ we can choose \[c'\left(\frac{\log n}{r}\right)^{1/5} \le 
k \le c''\left(\frac{\log n}{r}\right)^{1/5}.\]

Then for each $\varepsilon > 0$ there is a constant $c>0$ such that 
\eqref{eq:ineq} above holds for $n$ large enough.  
Earlier we assumed that $k\le \varepsilon\log n/\log 2-1$.  This holds for the
value of $k$ given above for $n$ large enough.  So the number of unit distances
from $\Gamma$ is less than $cn^{1+\varepsilon}$ for each $\varepsilon>0$.
\end{proof}

Performing a careful analysis of Erd\H{o}s' lower bound construction one can
show that all unit distances come from a group with rank at most $c\log
n/\log\log n$ for some $c>0$.  This group is generated by considering solutions
of the equation $x^2+y^2=p$ where $p$ is prime.  Using the prime number theorem
for arithmetic progressions we get the bound on such solutions and thus on the
rank.  For all the details see \cite{Schw13b}.  So Erd\H{o}s' construction
satisfies the conditions of Theorem~\ref{thm:unit}.  A similar approach could be
used for other types of lattices.  So all the best known lower bounds for the
unit distance problem have unit distances coming from a well structured group.
It would be interesting to see if any configuration of points with maximum unit
distances has such a structure.

\subsection{Sum-product estimates}
\label{subsec:sumProduct}

The theory of sum sets and product sets plays an important part in combinatorics and additive
number theory.  The goal of the field is to show that for any finite subset of 
the complex numbers either the sum set or the product set is large.

Formally, given a set $A\subset\mathbb{C}$, the sum set, denoted by $A+A$, and 
product set, denoted by $AA$, are \[A+A:=\{a+b : a,b \in A\},\qquad AA:=\{ab : 
a,b\in A\}.\]

The following long standing conjecture of Erd\H{o}s and Szemer\'edi 
\cite{Erdo83} has led to much work in the field.
\begin{conjecture}
   Let $\varepsilon>0$ and $A\subset \mathbb{Z}$ with $|A|=n$.  Then \[|A+A|+|AA|\ge
   Cn^{2-\varepsilon}.\]
\end{conjecture}
This conjecture is still out of reach.  The best known bound, which holds for 
real
numbers and not just integers, is $Cn^{4/3-o(1)}$ due to Solymosi
\cite{Soly09}. A similar bound was proved recently by Konyagin  and Rudnev in \cite{KR}.

Chang showed that when the product set is small the Subspace Theorem can be used 
to show that
the sum set is large \cite{Chan06}. The following reformulation of Chang's observation is due to Andrew Granville.
\begin{theorem}
   \label{thm:sumProduct}
   Let $A\subset \mathbb{C}$ with $|A|=n$.  Suppose $|AA| \le Cn.$  Then
   \[|A+A|\ge \frac{n^2}{2}+O_C(n).\]
\end{theorem}
We will present the proof of Theorem \ref{thm:sumProduct} below.  To use the Subspace Theorem we need a 
multiplicative subgroup with finite rank to work with.  The following lemma of Freiman provides 
this \cite{Frei73}.
\begin{lemma}[Freiman]\label{lem:freiman}
   Let $A\subset\mathbb{C}$.  If $|AA|\le C|A|$ then $A$ is a subset of a
   multiplicative subgroup of $\mathbb{C}^*$ of rank at most $r(C)$.
\end{lemma}
\begin{proof}[Proof of Theorem~\ref{thm:sumProduct}]
   We consider solutions of $x_1+x_2=x_3+x_4$ with $x_i\in A$.  A solution of 
   this equation corresponds to two pairs of elements from $A$ that give the 
   same element in $A+A$. Let us suppose that $x_1+x_2\neq 0$ (there are at
   most $|A|=n$ solutions of the equation $x_1+x_2= 0$ with $x_1,x_2\in A.$) 

   First we consider the solutions with $x_4=0$.  Then by rearranging we get 
   \begin{equation} \label{eq:three}
      \frac{x_1}{x_3}+\frac{x_2}{x_3}=1.
   \end{equation}
   By Lemma~\ref{lem:freiman} and Theorem~\ref{thm:ss3} there are at most 
   $s_1(C)$ solutions of $y_1+y_2=1$ with no subsum vanishing.  Each of these 
   gives at most $n$ solutions of \eqref{eq:three} since there are $n$ choices 
   for $x_3$.  There are only two solutions of $y_1+y_2=1$ with a vanishing 
   subsum, namely $y_1=0$ or $y_2=0$, and each of these gives $n$ solutions of 
   \eqref{eq:three}.  So we have a total of $(s_1(C)+2)n$ solutions of 
   \eqref{eq:three}.

   For $x_4\ne 0$ we get
   \begin{equation}\label{eq:four}
      \frac{x_1}{x_4}+\frac{x_2}{x_4}-\frac{x_3}{x_4}=1.
   \end{equation}
   Again by Freiman's Lemma and the Subspace Theorem, the number of solutions of 
   this with no vanishing subsum is at most $s_2(C)n$.  If we have a vanishing 
   subsum then $x_1=-x_2$ which is a case we excluded earlier or $x_1=x_3$ and
   then $x_2=x_4,$ or $x_2=x_3$ and then $x_1=x_4.$  So we get at most $2n^2$ 
   solutions of \eqref{eq:four} with a vanishing subsum (these are the $x_1+x_2=x_2+x_1$ identities.)

   So, in total, we have at most $2n^2+s(C)n$ solutions of $x_1+x_2=x_3+x_4$ 
   with $x_i\in A$.  Suppose $|A+A|=k$ and $A+A=\{\alpha_1, \dots, \alpha_k\}$.
   We may assume that $\alpha_1=0$.  Recall that we ignore sums $a_1+a_2=0$.
   Let \[P_i=\{(a,b)\in A\times A : a+b=\alpha_i\}.\]  Then \[\sum_{i=2}^k|P_i| 
   \ge n^2-n=n(n-1).\]  Also, a solution of $x_1+x_2=x_3+x_4$ corresponds to picking two 
   values from $P_i$ where $x_1+x_2=\alpha_i$.  Thus
      \begin{align*}
         2n^2 + s(C)n \ge \sum_{i=2}^k|P_i|^2\ge 
         \frac{1}{k-1}\left(\sum_{i=2}^k|P_i|\right)^2 \ge \frac{n^2(n-1)^2}{k-1}
      \end{align*}
      by the Cauchy-Schwarz inequality.  The bound for $k$ follows.
\end{proof}

A number of other combinatorial results follow from the Subspace Theorem.  We
give one more of these, from combinatorial geometry.  This is similar to a result
due to Chang and Solymosi \cite{Chan07}.     
Given two lines $L$ and $M$ we denote their point of intersection by $L\cap M$.
\begin{theorem} \label{thm:lines}
   Let $C>0$.  Then there exists $c>0$ such that for any $n+3$ lines $L_1, L_2, 
   L_3, M_1, \dots, M_n$ in $\mathbb{C}^2$, with $L_1\cap L_2, L_1\cap L_3$ and  
   $L_2\cap L_3$ distinct, if the number of distinct intersection points 
   $L_i\cap M_j, 1\le i\le 3, 1\le j\le n$, is at most $C\sqrt{n}$ then any line 
   $L \notin \{L_1, L_2, L_3\}$ has at least $cn$ distinct intersection points 
   $L\cap M_j, 1\le j\le n$.
\end{theorem}

There are many structure results similar to Theorem~\ref{thm:lines} in
discrete geometry.  These include Beck's Theorem \cite{Beck83}, a structure
theorem for lines containing many points of a cartesian product by Elekes
\cite{Elek97} and generalisations of this line theorem to surfaces by Elekes and
R\'onyai \cite{Elek00}, Elekes and Szab\'o \cite{Elek12} and Frank de Zeeuw and
the authors \cite{Schw13a}.  The proofs of these results used the 
Szemer\'edi-Trotter Theorem and techniques from commutative algebra and 
algebraic geometry.  These theorems have been used to prove various results
including a conjecture of Purdy about the number of distinct distances between
two sets of collinear points in the plane.  For more details see \cite{Elek02},
\cite{Elek99} and \cite{Schw13a}.

We do not prove Theorem~\ref{thm:lines} completely but only give a sketch of how
it follows from the Subspace Theorem.  We don't try to find an efficient
quantitative version here and we don't explain the refereed theorems in detail.
The techniques applied are standard methods in additive combinatorics. All the
details can be found in the book of Tao and Vu, "Additive Combinatorics"
\cite{Tao06}.

\medskip \noindent Apply an affine transformation which moves $L_1$ to the
$x$-axis, $L_2$ to the $y$-axis, and $L_3$ to the horizontal line $y=1$.  The
three lines have distinct intersection points thus such a transformation exists.
Let us denote the $x$-coordinates of $L_1\cap M_i$ and $L_3\cap M_j$ by $x_i$
and $y_j$ respectively. The two sets of $x$-coordinates are denoted by $X$ and
$Y.$ Define a bipartite graph with vertices given by the intersection points of
lines $M_i$ with $L_1$ and $L_3$ (with vertex sets $X$ and $Y$ without
multiplicity.)  Two points are connected by an edge in the graph if they are
connected by a line $M_j.$ This is a bipartite graph on at most $C\sqrt{n}$
vertices with $n$ edges.  Using Szemer\'edi's Regularity Lemma one can find a
regular (random-like) bipartite graph, $G,$ with at least $c'n$ edges and vertex
sets $V_1\subset X$ and $V_2\subset Y.$ If $M_i \cap L_2$ is the point
$(0,\alpha)$ then $x_i/y_i=\alpha/(1-\alpha),$ or equivalently $x_i=\alpha
y_i/(1-\alpha).$ The Balog-Szemer\'edi Theorem and Freiman's Lemma imply that
there are large subsets $X'\subset V_1$ and $Y'\subset V_2$ so that $X'$ and
$Y'$ are subsets of a multiplicative subgroup of $\mathbb{C}^*$ of rank at most
$r(C)$. As $G$ is regular, the subgraph spanned by $X',Y'$ still has at least
some $c''n$ edges.  We show that the lines represented by these $c''n$ edges
cannot have high multiplicity intersections outside of $L_1,L_2, L_3.$ If
$(a,b)$ is a point of $M_i$ connecting two points of $X'$ and $Y'$ then
$(a-x_i)/(a-y_i)=b/(1-b),$ which gives the solution $(x_i,y_i)$ to the equation
$cx+dy=1$ if $a\neq 0, b\neq 0,1.$ Here $c, d$ depend on $a$ and $b$ only. As
$x_i$ and $y_i$ are from a multiplicative group of bounded rank, we have a
uniform bound, $B,$ on the number of lines between $X'$ and $Y'$ which are
incident to $(a,b).$  There are $c''n$ lines connecting at most $C\sqrt{n}$
points. No more than $C\sqrt{n}/2$ of them might be parallel to any given line.
Any line intersects at least $c''n-C\sqrt{n}$ of them.  Any intersection point
outside of lines $L_1, L_2,$ and $L_3$ is incident to at most $B$ lines, so
there are at least $cn$ distinct intersection points  $L\cap M_j, 1\le j\le n$
with any other line.

\medskip
\noindent 
We are unaware of any proof of this fact without the Subspace Theorem.

\bibliographystyle{plain}
\bibliography{refs}

\end{document}